\documentclass[a4paper,11pt,leqno]{amsart}
\usepackage{amsmath,amsfonts,amsthm,amssymb,dsfont}
\usepackage[alphabetic]{amsrefs}
\usepackage{mathrsfs}
\usepackage{enumerate, xspace}
\usepackage{graphics}
\usepackage{hyperref}
\usepackage[usenames, dvipsnames]{color}

\long\def\symbolfootnote[#1]#2{\begingroup
\def\thefootnote{\fnsymbol{footnote}}\footnote[#1]{#2}\endgroup}

\newtheorem{thm}{Theorem}[section]
\newtheorem{cor}[thm]{Corollary}

\newtheorem{lem}[thm]{Lemma}
\newtheorem{prop}[thm]{Proposition}

\theoremstyle{definition}

\newtheorem{rem}[thm]{Remark}

\newtheorem{question}[thm]{Question}

\newcommand{\norma}{{N}}
\newcommand{\centra}{C}
\newcommand{\la}{\langle}
\newcommand{\ra}{\rangle}

\newcommand{\Aut}{\mathrm{Aut}}
\newcommand{\Fix}{\mathrm{Fix}}
\newcommand{\Out}{\mathrm{Out}}
\newcommand{\Inn}{\mathrm{Inn}}
\newcommand{\Pc}{\mathrm{Pc}}
\newcommand{\pc}{\mathrm{Pc}}
\newcommand{\proj}{\mathrm{proj}}
\newcommand{\Stab}{\mathrm{Stab}}
\newcommand{\Res}{\mathrm{Res}}
\newcommand{\inv}{^{-1}}

\newcommand{\N}{\mathbb{N}}
\newcommand{\Z}{{\mathbb Z}}

\newcommand{\bG}{{\bar{G}}}

\newcommand{\bx}{{\bar{x}}}

\newcommand{\bh}{{\bar{h}}}

\DeclareMathOperator{\rank}{rank}
\newcommand{\gen}[1]{\left\langle#1\right\rangle}

\begin{document}
\title[{On conjugacy separability of some Coxeter groups}]{{On conjugacy separability of some Coxeter groups and parabolic-preserving automorphisms}}

\author[P.-E. Caprace]{Pierre-Emmanuel Caprace}
\address[Pierre-Emmanuel Caprace]{Universit\'e catholique de Louvain, IRMP, Chemin du Cyclotron 2, box L7.01.02, 1348 Louvain-la-Neuve, Belgium}
\email[Pierre-Emmanuel Caprace]{pe.caprace@uclouvain.be}

\author[A. Minasyan]{Ashot Minasyan}
\address[Ashot Minasyan]{School of Mathematics,
University of Southampton, Highfield, Southampton, SO17 1BJ, United Kingdom.}
\email[Ashot Minasyan]{aminasyan@gmail.com}
\thanks{P-E.\,C. was partly supported by FNRS grant F.4520.11 and the European Research Council (grant \#278469). A.\,M. was partially supported by the EPSRC grant EP/H032428/1.}



\begin{abstract} We prove that even Coxeter groups, whose Coxeter diagrams contain no $(4,4,2)$ triangles, are conjugacy separable.
In particular, this applies to all right-angled Coxeter groups or word hyperbolic even Coxeter groups.
{For an arbitrary Coxeter group $W$, we also study the relationship between Coxeter generating sets that give rise to the same collection of parabolic subgroups.
As an application we show that if an automorphism of $W$ preserves the conjugacy class of every
sufficiently short element then it is inner.} We then derive consequences for the outer automorphism groups of Coxeter groups.
\end{abstract}
\maketitle

\section{Introduction}
A group $G$ is said to be \textbf{conjugacy separable} if for any two non-conjugate elements $x,y \in G$ there is a homomorphism from $G$ to a finite group $M$ such that
the images of $x$ and $y$ are not conjugate in $M$. Conjugacy separability  can be restated by saying that each conjugacy class $x^G:=\{gxg^{-1} \mid g \in G\}$
is closed in the profinite topology on $G$. If $G$ is residually finite, this also equivalent to the equality $x^G= x^{\widehat G} \cap G$ in $\widehat G$ for all $x \in G$,
where $\widehat{G}$ denotes the profinite completion of $G$.

Conjugacy separability is a classical notion from Combinatorial Group Theory. Originally it was introduced by Mostowski \cite{Most},  who suggested the first application of this property by proving that
a finitely presented conjugacy separable group has solvable conjugacy problem (see also Malcev's work \cite{Malc-cs}).
Presently the following classes of groups are known to be conjugacy separable:
virtually free groups (Dyer \cite{J-Dyer}), virtually surface groups (Martino \cite{Martino}),
virtually polycyclic groups (Remeslennikov \cite{Rem}; Formanek \cite{Form}),
finitely presented residually free groups (Chagas and Zalesskii \cite{Chag-Zal}), right angled Artin groups (Minasyan \cite{Min-RAAG}), non-uniform arithmetic lattices in
${\rm SL}_2(\mathbb{C})$ (Chagas and Zalesskii \cite{C-Z_bianchi}) .

While conjugacy separability is a natural amplification of residual finiteness, it is usually much harder to establish. One of the difficulties comes from the fact that, in general, conjugacy separability is not stable under passing
to finite index subgroups or overgroups (see \cites{Goryaga,Chag-Zal,M-M}).  In view of this Chagas and Zalesskii call a group $G$ \textbf{hereditarily conjugacy separable} if every finite index subgroup of $G$ is conjugacy separable.
Recent theorems {due to} Haglund and Wise \cites{Haglund-Wise_Special,Haglund-Wise_Cox}, Wise \cites{Wise-qc_hierarchy} and Agol \cite{Agol} show that many naturally occurring groups possess finite index subgroups
that embed into right angled Artin groups as virtual retracts. If $G$ is such a group, then, by the work {of} the second author \cite{Min-RAAG}, $G$ contains a {hereditarily} conjugacy separable subgroup of finite index.

The goal of the present work is to study conjugacy separability and related properties for Coxeter groups.
Recall that a \textbf{Coxeter group} is a group  $W$ given by the presentation
\begin{equation}\label{eq:Cox-def}
W=\langle s_1,\dots,s_n\,\|\, (s_is_j)^{m_{ij}}=1, \,\mbox{ for all } i,j  \rangle,\end{equation}
where $M:=(m_{ij})$
is a symmetric $n \times n$ matrix, whose entries satisfy the following conditions:
$m_{ii}=1$ for every $i=1,\dots,n$, $m_{ij} \in \N \sqcup \{\infty\}$ (if $m_{ij}=\infty$, then it is understood that no relation on the product $s_is_j$ is added in the Coxeter presentation) and $m_{ij} \ge 2$
whenever $1 \le i<j \le n$.
The set $S=\{s_1,\dots,s_n\}$ is called the \textbf{Coxeter generating set} for $W$,  $M$ is called the \textbf{Coxeter matrix}  and $n=|S|$ is called the \textbf{rank} of $W$.

Each Coxeter group $W$ is associated with a {\textbf{{(free)} Coxeter diagram}}, which is a labeled graph whose vertex set is indexed by the generators  $\{ s_1,\dots,s_n\}$ such that
between vertices corresponding to {distinct generators $s_i$ and $s_j$, there is an edge labeled by $m_{ij}$  if and only if
$m_{ij} \neq \infty$.} 
The Coxeter group $W$ is said to be \textbf{even} if all  {non-diagonal} entries in the corresponding Coxeter matrix $M$ are either even integers or $\infty$.
{The group $W$ is \textbf{right-angled} if $m_{ij} \in \{2,\infty\}$ whenever $i \neq j$.}
Coxeter groups have been an object of intensive study for many years. For background and basic properties of Coxeter groups the reader is referred to  \cite{Davis}.

In \cite{Niblo-Reeves} for any given Coxeter group $W$,  Niblo and Reeves construct a CAT(0) cube complex $X$ on which $W$ acts properly by isometries.
A combination of Theorem 1.2 from \cite{Haglund-Wise_Cox} and  Corollary 2.2 from \cite{Min-RAAG} implies that every Coxeter group $W$, which acts cocompactly on its
Niblo-Reeves cube complex, possesses a {hereditarily conjugacy} separable subgroup of
finite index.  Therefore it is natural to ask the following question:

\begin{question}\label{q:1} Is every finitely generated Coxeter group conjugacy separable?
\end{question}

Our first result provides a positive answer to the above question for a large class of even Coxeter groups:

\begin{thm}\label{thm:some_even_Cox-cs} Suppose that $W$ is an even Coxeter group of finite rank such that its Coxeter diagram has no $(4,4,2)$-triangles
(i.e., no subdiagrams of type $\widetilde{B}_2$).

Then $W$ is conjugacy separable.
\end{thm}

Even Coxeter groups covered  by Theorem \ref{thm:some_even_Cox-cs} are precisely the ones that act cocompactly on their Niblo-Reeves cube complexes. This follows from
a result of the first author and M\"uhlherr \cite{Caprace-Muhlherr} stating that the action of  $W$ on its Niblo-Reeves cubulation is cocompact if and only if its  Coxeter diagram has no irreducible affine subdiagrams
of rank at least $3$. By the classification of irreducible affine Coxeter groups, in the case when $W$ is even the latter condition is equivalent to the absence of $(4,4,2)$-triangles in the Coxeter diagram of $W$.
In particular, Theorem \ref{thm:some_even_Cox-cs} applies if $W$ is right-angled or if $W$ is even and word hyperbolic.

The proof of Theorem \ref{thm:some_even_Cox-cs} basically splits into two parts.
In the first part we employ a criterion of Chagas and Zalesskii \cite{C-Z_bianchi} to show that \emph{essential} elements in $W$ {(i.e., elements not contained in any proper parabolic subgroup)}
are conjugacy distinguished. {This relies on the fact that} $W$
contains a hereditarily conjugacy separable subgroup of finite index, as discussed above.
In the second part, to deal with non-essential elements we introduce a new criterion
(Lemma \ref{lem:retr_crit}), which works because standard parabolic subgroups in even Coxeter groups are retracts. In particular we prove that finite order
elements are conjugacy distinguished in any even Coxeter group (Proposition \ref{prop:f_o_in_Cox-cd}).

{Another standard application of conjugacy separability was discovered by Grossman
\cite{Grossman}, who
proved that for a finitely generated conjugacy separable group $G$, the group of outer automorphisms $\Out(G)$ is residually finite, provided every pointwise inner automorphism of $G$ is inner.
Recall that an automorphism $\alpha$ of a group $G$ is called \textbf{pointwise inner} if $\alpha(g)$ is conjugate to $g$ for every $g \in G$.
Presently it is unknown whether the outer automorphism group of every finitely generated  Coxeter group is residually finite.
That some (and conjecturally all) Coxeter groups  are conjugacy separable therefore motivates the question whether pointwise inner automorphisms
of Coxeter groups are necessarily inner. {A positive answer for \emph{all} finitely generated Coxeter groups is provided by Corollary~\ref{cor:CoxPtwiseInner} below. This will be deduced from a study of
automorphisms {that preserve} parabolic subgroups. In order to present a precise formulation, we first recall that
an automorphism of a Coxeter group is called \textbf{inner-by-graph} if it maps a Coxeter generating set $S$ to a (setwise) conjugate of itself.
Such an automorphism is thus a composition of an inner automorphism with a \textbf{graph automorphism}, i.e., an automorphism which stabilises the Coxeter generating set $S$.


\begin{thm}\label{thm:CoxInner}
Let $W$ be a finitely generated Coxeter group with Coxeter generating set $S$, and let  $\alpha \in \Aut(W)$ be an automorphism.

Then $\alpha$ is inner-by-graph if and only if it satisfies the following two conditions:
\begin{enumerate}[(1)]
\item $\alpha$ maps every parabolic subgroup to a parabolic subgroup.

\item For all $s, t \in S$, such that $st$ has finite order {in $W$}, there is a pair $s', t' \in S$ such that $\alpha(st)$ is conjugate to $s't'$.
\end{enumerate}
\end{thm}

{Theorem \ref{thm:CoxInner} will follow from} Proposition~\ref{prop:MainTech} and Theorem~\ref{thm:ParabCompat} below.
The condition (2) in Theorem~\ref{thm:CoxInner} can be interpreted geometrically: it means that the reflections $s$ and $t$ are mapped by $\alpha$ to a
pair of reflections such that the \emph{angle} between their fixed walls is preserved. In the terminology recalled in Section~\ref{sec:ParabPreserv}
below, we say that the Coxeter generating sets $S$ and $\alpha(S)$ are \emph{angle-compatible} {(cf. Lemma \ref{lem:AngleCompat})}.
For a thorough study of the relation of angle-compatibility, we refer to \cite{MM08}.

It is easy to see that condition (2) is necessary for $\alpha$ to be inner-by-graph: examples illustrating that matter of fact may be found among finite dihedral groups.
It turns out, however, that is $W$ is \textbf{crystallographic}, i.e., if $m_{ij} \in \{2, 3, 4, 6, \infty\}$ for all $i \neq j$, then  condition (2) is automatically satisfied. In particular we obtain

\begin{cor}\label{cor:cryst}
Let $W$ be a finitely generated crystallographic Coxeter group. Then an automorphism $\alpha \in \Aut(W)$ is inner-by-graph if and only if $\alpha$ maps every parabolic subgroup to  a parabolic subgroup.
\end{cor}


We shall also see that an automorphism of $W$ preserving the conjugacy class of every element of small word length {(with respect to $S$)}
also satisfies the conditions of Theorem~\ref{thm:CoxInner}. {In fact, in such a case one can even exclude graph automorphisms, thereby yielding the following corollary.}

\begin{cor}\label{cor:SmallWords}
Let $W$ be a finitely generated Coxeter group with Coxeter generating set $S$, and let  $\alpha \in \Aut(W)$ be an automorphism.
Suppose that $\alpha(w)$ is conjugate to $w$ for all elements $w$ that can be written as  products of pairwise distinct generators  (in particular the word length of such elements is bounded above by $|S|$).

Then $\alpha$ is inner.
\end{cor}

The following consequence of Corollary~\ref{cor:SmallWords} is immediate:
}

\begin{cor}\label{cor:CoxPtwiseInner}
Every pointwise inner automorphism of a finitely generated Coxeter group is inner.
\end{cor}

Combining Corollary \ref{cor:CoxPtwiseInner}, Theorem \ref{thm:some_even_Cox-cs} together with the theorem of Grossman \cite{Grossman} mentioned above we obtain  {the following}.

\begin{cor}\label{cor:Out_rf}
Assume that $W$ is a finitely generated even Coxeter group whose Coxeter diagram contains no $(4,4,2)$-triangles. Then $\Out(W)$ is residually finite.
\end{cor}

}



\section{Criteria for conjugacy separability of finite index overgroups}
Let $G$ be a group. The \textbf{profinite topology} $\mathcal{PT}(G)$ on $G$ is the topology whose basic open
sets are cosets of finite index normal subgroups in $G$. It follows that every finite index subgroup $K \leqslant G$ is both
closed and open in  $\mathcal{PT}(G)$, and $G$, equipped with  $\mathcal{PT}(G)$, is
a topological group (that is, the group operations are continuous with respect to this
topology). This topology is Hausdorff if and only if the intersection of all finite index
normal subgroups is trivial in $G$. In this case $G$ is said to be \textbf{residually finite}.

A subset $S$ of a group $G$ is said to be \textbf{separable} if $S$ is closed in the profinite topology on $G$.
An element $x \in G$ is \textbf{conjugacy distinguished} if $x^G$ is separable in $G$
(if $S\subseteq G$ and $x \in G$ then $x^S:=\{s^{-1}xs \mid s \in S\} \subseteq G$).
This is equivalent to the following statement: for every element $y \in G \setminus x^G$ there
exist a finite group $F$ and a homomorphism $\alpha:G \to F$ such that $\alpha(y) \notin \alpha(x)^F$ in $F$.

\begin{rem} \label{rem:subspace_top}
{Let $G$ be a group and $H \leqslant G$ a subgroup.} Then the subspace topology on $H$ induced from  $\mathcal{PT}(G)$ is weaker than the $\mathcal{PT}(H)$
(in other words, if a subset $S \subseteq H$ is closed in the subspace topology {induced from $\mathcal{PT}(G)$}, then it is also closed in  $\mathcal{PT}(H)$). If, in addition $H$ has finite index in $G$,
then this induced subspace topology coincides with  $\mathcal{PT}(H)$, i.e., a subset $S \subseteq H$ is closed in  $\mathcal{PT}(H)$ if and only if it is closed in  $\mathcal{PT}(G)$.
\end{rem}

The following criterion was discovered by Chagas and Zalesskii in \cite[Prop. 2.1]{C-Z_bianchi}; we will present a proof here for the sake of completeness.
\begin{prop}\label{prop:C-Z-crit} Let $H$ be a normal subgroup of a group $G$ such that $|G:H|=m<\infty$ and let $x \in G$. Suppose that $H$ is hereditarily conjugacy separable
and the $G$-centralizer $C_G(x^m)$ of $x^m \in H$  satisfies the following conditions:
\begin{itemize}
	\item[(i)] $C_G(x^m)$ is conjugacy separable;
	\item[(ii)] each finite index subgroup $N$, of  $C_G(x^m)$, is separable in $G$.	
\end{itemize}
Then the conjugacy class $x^G$ is separable in $G$.
\end{prop}

To prove the proposition we will need the following statement, which is a special case of Lemma 2.7 from \cite{Min-RAAG}:

\begin{lem}\label{lem:for_a_given_K_sep_cc->CCG} Let $G$ be a group.
Suppose that $h \in G$, $K\lhd G$ and $|G:K|<\infty$. If the
subset $h^{K}$ is separable in $G$, then there is a finite index normal subgroup $L \lhd G$
such that $L \leqslant K$ and $ C_{G/L} (\psi(h))  \subseteq \psi\left (C_G(h) K\right)$
in $G/L$, where $\psi:G \to G/L$ is the natural epimorphism.
\end{lem}

\begin{proof} By the assumptions, $G=\bigsqcup_{i=1}^k z_iK$ for some $z_1,\dots,z_k \in G$. Renumbering the elements
$z_i$, if necessary, we can suppose that there is $l \in \{0,1,\dots,k\}$ such that  $z_i^{-1} h z_i \notin h^{K}$ whenever
$1 \le i \le l$, and  $z_j^{-1} h z_j \in h^{K}$ whenever
$l+1 \le j \le k$.

By the assumptions, there exists a finite index normal subgroup $L \lhd G$ such that
$z_i^{-1} h z_i \notin h^{K} L$ for $1\le i \le l$.
Moreover, after replacing $L$ with $L \cap K$, we can assume that $L \leqslant K$.

Let $\psi$ be the natural epimorphism from $G$ to $G/L$ and let $\bG:=G/L$, $\bh:=\psi(h) \in \bG$. Consider any element $\bar x \in C_\bG(\bh)$.
Then $\bar x = \psi(x)$ for some $x \in G$, and
$\psi(x^{-1}h x)=\psi(h)$ in $G/L$, i.e., $x^{-1} h x \in hL$ in $G$. As we know, there is
$i \in \{1,\dots,k\}$ and $y \in {K}$ such that $x = z_i y$. Consequently,
$z_i^{-1} h z_i \in y h Ly^{-1} = y h y^{-1}L \subseteq h^{K} L$. Hence, $i \ge l+1$, that is,
$z_i^{-1}h z_i = u h u^{-1}$ for some $u \in {K}$.

Thus $z_i u  \in C_G(h)$ and $x=z_i y = ( z_i u) (u^{-1}y) \in C_G(h) K$. Therefore we proved that $\bx \in \psi(C_G(h)K)$ in $G/L$ for every $\bx \in C_\bG(\bh)$.
This yields the desired inclusion: $ C_{\bG} (\bh)  \subseteq \psi\left (C_G(h) K\right)$ in $\bG=G/L$.
\end{proof}

\begin{proof}[Proof of Proposition \ref{prop:C-Z-crit}]
Consider any $y \in G$ such that $y \notin x^G$.
Suppose, first, that $y^m \notin (x^m)^G$, {where $m = |G:H|$}. Write $G=\bigsqcup_{i=1}^n g_iH $ for some $g_1,\dots,g_n \in G$ and observe that
$y^m \notin (x^m)^{g_iH}= (g_i^{-1}x^mg_i)^H$ for all $i=1,\dots,n$. Since $H$ is conjugacy separable and $g_i^{-1}x^mg_i \in H$, the set $(g_i^{-1}x^mg_i)^H$ is separable in $H$
for every $i$. Hence the set $(x^m)^G=\bigcup_{i=1}^n \left((x^m)^{g_iH}\right)$ is closed in the profinite topology on $H$, implying that it is also separable in $G$ (see Remark \ref{rem:subspace_top}).
Therefore, there is a finite group $F$
and a homomorphism $\alpha:G \to F$ such that $\alpha(y^m)$ is not conjugate to $\alpha(x^m)$ in $F$. Consequently $\alpha(y) \notin \alpha(x)^{F}$, as required.

Thus we can assume that $y^m=g^{-1}x^mg$ for some $g \in G$. Moreover, upon replacing $y$ with $gyg^{-1} \in G$, we can further suppose that $y^m=x^m$. Then $x,y \in C_G(h)$, where $h:=x^m \in H$, and
by conjugacy separability of $C_G(h)$ we can find a finite index normal subgroup $N \lhd C_G(h)$ such that $y \notin x^{C_G(h)}N$ in $G$.
Write $C_G(h)=N \sqcup \bigsqcup_{j=1}^l Nf_j$ for some $f_1,\dots,f_l \in C_G(h)\setminus N$. According to the assumption (ii), there is a finite index
normal subgroup $K \lhd G$ such that $f_j \notin  NK$ for every $j=1,\dots,l$, hence $K \cap Nf_j=\varnothing$ for all $j$. It follows that $K \cap C_G(h) \leqslant N$.

After replacing  $K$ with $H \cap K$, we can assume that $K \leqslant H$. Then $|H:K|<\infty$ and our assumptions imply that $h^K$ is separable in  $H$, and, hence, in $G$.
Therefore, by Lemma \ref{lem:for_a_given_K_sep_cc->CCG}, there is $L \lhd G$ such that $L \leqslant K$ and $ C_{G/L}(\psi(h)) \leqslant \psi\left(C_G(h)K\right)$, where $\psi:G \to G/L$
denotes the natural epimorphism.

We claim that $\psi(y) \notin \psi(x)^{G/L}$ in $G/L$. Indeed, suppose, on the contrary, that there is $u \in G$ such that $\psi(y)=\psi(u)^{-1} \psi(x) \psi(u)$.
Then $\psi(u) \in C_{G/L}(\psi(x^m))$, hence $u \in C_{G}(h)KL=C_G(h)K$. It follows that $y\in u^{-1}xuL \subseteq	x^{C_G(h)}K$ in $G$. But since $x,y \in C_G(h)$,
the latter means that $y \in x^{C_G(h)}(K \cap C_G(h)) \subseteq x^{C_G(h)}N$, contradicting the choice of $N$.

Thus we found a finite quotient-group of $G$ such that the images of $x$ and $y$ are not conjugate in this quotient; therefore $x^G$ is separable in $G$.
\end{proof}

Proposition \ref{prop:C-Z-crit} was used by Chagas and Zalesskii to show that certain torsion-free extensions of hereditarily conjugacy separable groups are conjugacy separable (see \cite{C-Z_bianchi}).
However, in order to deal with torsion we need to find a different criterion.

Let $G$ be a group and let $A$ be a subgroup of $G$.
Recall that an endomorphism $\rho_A: G \to G$ is called a \textbf{retraction} of $G$ onto $A$ if
$\rho_A(G)=A$ and $\rho_A(h)=h$ for every $h \in A$. In this case $A$ is said to be a \textbf{retract} of $G$.
Note that $\rho_A \circ \rho_A=\rho_A$. 

Assume that $A$ and $B$ are two retracts of a group $G$ and $\rho_A,\rho_B\in {\rm End}(G)$ are the corresponding retractions. We will say
$\rho_A$ {\it commutes} with $\rho_B$ if they commute as elements of the monoid of
endomorphisms $\mathrm{End}(G)$, i.e., if $\rho_A ( \rho_B(g))=\rho_B ( \rho_A (g))$ for all $g \in G$.

\begin{rem}[Rem. 4.2 in \cite{Min-RAAG}]\label{rem:inter_comm_retr} If the retractions $\rho_A$ and $\rho_B$ commute then
$\rho_A(B)=A \cap B= \rho_B(A)$ and
the endomorphism $\rho_{A \cap B}:=\rho_A \circ \rho_B=\rho_B \circ \rho_A$
is a retraction of $G$ onto $A \cap B$.
\end{rem}

Indeed, obviously the restriction of $\rho_{A\cap B}$ to ${A\cap B}$ is the identity map.
And $\rho_{A\cap B}(G) \subseteq \rho_A(G) \cap \rho_B(G)=A \cap B$, hence $\rho_{A\cap B}(G)=A \cap B$.
Consequently $\rho_A(B)=\rho_A(\rho_B(G))=\rho_{A \cap B}(G)=A \cap B$. Similarly, $\rho_B(A)=A \cap B$.

\begin{lem} \label{lem:retr_crit} Suppose that $A,B \leqslant G$ are retracts of $G$ such that the corresponding retractions $\rho_A,\rho_B \in {\rm End}(G)$ commute.
Then for arbitrary elements $x \in A$ and $y \in B$, $x$ is conjugate to $y$ in $G$ if and only if the following three conditions hold:
\begin{enumerate}
	\item\label{it:1} $\rho_A(y) \in x^A$ in $A$;
	\item\label{it:2} $\rho_B(x) \in y^B$ in $B$;
	\item\label{it:3} $\rho_{A\cap B}(y) \in \rho_{A\cap B}(x)^{A \cap B}$ in  $A\cap B$.
\end{enumerate}
\end{lem}

\begin{proof} Suppose that $y=g^{-1}xg$ for some $g \in G$. Applying $\rho_A$ to both sides of this equality we achieve $\rho_A(y)=\rho_A(g)^{-1} x \rho_A(g)$, thus $\rho_A(y) \in x^A$.
Similarly, $\rho_B(x) \in y^B$. Finally, \eqref{it:3} follows after applying $\rho_{A\cap B}$ to both sides of the above equality.

Assume, now, that the conditions \eqref{it:1}--\eqref{it:3} hold. Note that $\rho_{A\cap B}(x)=\rho_B(\rho_A(x))=\rho_B(x)$ as $x \in A$; similarly, $\rho_{A\cap B}(y)=\rho_A(y)$.
Then $x$ is conjugate to $\rho_A(y)=\rho_{A\cap B}(y)$, which is conjugate to $\rho_{A\cap B}(x)=\rho_B(x)$, which is conjugate to $y$ in $G$. Since
conjugacy is a transitive relation we can conclude that $y \in x^G$.
\end{proof}


\section{Parabolic subgroups and parabolic closures in Coxeter groups}
In this section we collect some of the basic facts about parabolic subgroups of Coxeter group that will be used in the rest of the paper.

Let $W$ be a Coxeter group with a fixed finite Coxeter generating set $S$. In this section we will remind some terminology and basic properties of $W$ and its parabolic subgroups.
A \textbf{reflection} of $W$ is an element conjugate to some $s \in S$.
Given $J \subseteq S$, we set $W_J = \la J \ra$.
{A subgroup of the form $W_J$ for some $J \subseteq S$ is called a \textbf{standard parabolic subgroup} of
$W$. {It is a standard fact that  $W_J$  is itself a Coxeter group with Coxeter generating set $J$.} A subgroup $P$ is called \textbf{parabolic} if it is conjugate to some standard parabolic subgroup $W_J$. }
The \textbf{rank} $\rank(P)$ of that parabolic subgroup is the cardinality of $J$.

{The following basic property of parabolic subgroups is crucial.

\begin{lem}\label{lem:ParabHered}
Let $P,Q$ be two parabolic subgroups of a Coxeter group $W$. Then  $P \cap Q$ is a parabolic subgroup with respect to the Coxeter group $Q$ {(and the natural Coxeter generating set of $Q$ coming  from $S$)}.
\end{lem}

\begin{proof}
We recall that the Cayley graph of a Coxeter group may be viewed as a \textbf{chamber system}; this fact, as well as a basic introduction
to chamber systems, can be found in \cite{Weiss}.
We recall that parabolic subgroups in a Coxeter group are exactly the stabilisers of the residues.
Given parabolic subgroups $P \leqslant Q$, let $R_P$ and $R_Q$ be the residues whose stabilisers are precisely $P$ and $Q$.
Then the combinatorial projection $R'_P = \proj_{R_Q}(R_P)$ of $R_P$ on $R_Q$ is a residue
stabilised by $P \cap Q$ {(see \cite[Prop.~2.29]{TitsLN})}.
Moreover, properties of the combinatorial projection imply that if a reflection stabilises $R'_P$, then it also stabilises $R_P$.
It follows that the stabiliser of $R'_P$ also stabilises $R_P$, since it is generated by reflections. Finally, since $R'_P$ is contained in $R_Q$, the
stabiliser of $R'_P$ is also contained in $Q$ 
This shows that the stabiliser of $R'_P$ equals $P \cap Q$. Thus $P \cap Q$ is a parabolic
subgroup in the Coxeter group $Q$, as claimed.
\end{proof}

}

Lemma~\ref{lem:ParabHered} implies that any intersection of parabolic subgroups is a parabolic subgroup.
In particular any subset $H$ of $W$ is contained in a unique minimal parabolic subgroup, called the \textbf{parabolic closure} of $H$.
Moreover, if $P , Q$ are parabolic subgroups such that $P$ is properly contained in $Q$, then the rank of $P$ is strictly smaller than the rank of
$Q$ ({for the definition and basic properties of parabolic closures,} see \cite[\S2.1]{Krammer}).

\begin{lem}\label{lem:PcFinite}
Let $H \leqslant W$ be a finite subgroup generated by $n$ reflections. Then $\Pc(H)$ is a finite parabolic subgroup of rank~$\leq n$.
\end{lem}
\begin{proof}
The fact that $\Pc(H)$ is finite is well-known, see \cite[Ch.~V, \S4, Exercice~2.d]{Bourbaki}. Now it suffices to show that in a finite Coxeter group $W$, a
reflection subgroup generated by $n$ reflections is contained in a parabolic subgroup of rank~$n$. Let $\Sigma$ be the geometric realization of
the Coxeter complex of $W$ (see \cite[Ch.~2]{TitsLN} or \cite[Ch.~3]{AB} for Coxeter complexes). Each reflection fixes pointwise a hyperplane of the sphere $\Sigma$.
Thus $H$ fixes a subcomplex
$\Sigma'$ of codimension~$d \leq n$. Let $\sigma \subset \Sigma'$ be a simplex of codimension $d$. Then $H $ is contained in $P = \Stab_W(\sigma)$
and $P$ is a parabolic subgroup of rank $d \leq n$.
\end{proof}

{
Let  $J \subseteq S$. We set $J^\perp = \{s \in S \setminus J \; | \; sj = js \text{ for all } j\in J\}$.
The set $J$ is called \textbf{spherical} if $W_J$ is finite. The set} $ J \subseteq S$ is called \textbf{irreducible} if for every non-empty subset $I \subsetneq J$, we have $J \not \subset I \cup I^\perp$;
equivalently the parabolic subgroup $W_J$ does not split   as a direct product of proper parabolic subgroups. {It is a fact that if {an infinite Coxeter group $W$ admits an irreducible Coxeter generating set $S$, then any
other Coxeter generating set of $W$ is also irreducible.} (If $W$ is finite, this is, however, not the case, since a dihedral group of order $12$ is the direct product of a group of order $2$ and a dihedral group of order $6$.)
Thus, in that case, it makes sense to say that $W$ itself is \textbf{irreducible}.}

\begin{lem}\label{lem:Centra}
Let $J \subseteq S$ be irreducible and non-spherical.

\begin{enumerate}[(i)]
\item $\norma_W(W_J) = W_{J \cup J^\perp}$ and $\centra_W(W_J) = W_{J^\perp}$.

\item If $w Jw\inv \subset S$ for some $w \in W$, then $wJw\inv = J$.

\item If  $J^\perp = \varnothing$, then every parabolic subgroup of $W$ containing $W_J$ is standard.

\end{enumerate}
\end{lem}

\begin{proof}
For (i) and (ii), see \cite{Deodhar} or \cite[\S 3.1]{Krammer}. Assertion (iii) is well known to the experts and can be deduced from (i).
By lack of an appropriate reference, we provide a proof. To this end, we  view the Cayley graph $X$ of $W$ with respect to $S$  as a \emph{chamber system}
(see \cite[\S 5.2]{AB} for the definition of chamber systems and the associated terminology).

For each $I \subseteq S$, the parabolic subgroup $W_I$ is the stabiliser in $W$ of the $I$-residue of $X$ containing the base chamber $1$, which is denoted by $\Res_I(1)$.

Let now $R$ and $R'$ be two residues whose stabiliser in $W$ is $W_J$. For every wall $\mathcal W$ crossed by a minimal gallery joining a chamber in $R$ to its projection to $R'$,
the wall $\mathcal W$ does not cross $R'$ (by properties of the projection) and, hence, the associated reflection $r_{\mathcal W}$ does not stabilise $R'$.
Since $R$ and $R'$ have the same stabiliser, we infer that $\mathcal W$ does not cross $R$ either. This proves that every wall crossed by a gallery
joining a chamber in $R$ to its projection to $R'$, separates $R$ from $R'$. It follows that such a wall  $\mathcal W$ is contained in a bounded neighborhood
of $R$. Therefore the reflection $r_{\mathcal W}$ commutes with $W_J$ by \cite[Lemma~2.20]{CaMa}. From (i) and the hypothesis that $J^\perp$ is empty, we
infer that there is no wall separating $R$ from $R'$. In other words $\Res_J(1)$ is the unique residue in $X$ whose stabiliser is $W_J$.

Let now $P$ be a parabolic subgroup containing $W_J$. Then $P$ is the stabiliser of some residue $R$. Since $W_J \leqslant P$ it follows that $R$ contains a residue
whose stabiliser is $W_J$. Thus $R$ contains $\Res_J(1)$ by what we have just proved. It follows that $R$ is of the form $R = \Res_{J \cup J'}$ for some
$J' \subseteq S \setminus J$ which implies that $P$ is indeed standard.
\end{proof}

\begin{lem}\label{lem:CoxeterElement}
Let $J = \{s_1, \dots, s_n\} \subseteq S$ and denote by $w = s_1 s_2 \dots s_n$ the product of all elements of $J$ (ordered arbitrarily). Then $\Pc(w) = W_J$.
\end{lem}

\begin{proof}
See Theorem~3.4 in \cite{Paris} or Corollary~4.3 in \cite{CapraceFujiwara}.
\end{proof}

{

\section{Conjugacy separability in Coxeter groups}
Let $W$ be an even Coxeter group with a Coxeter generating set $S$. Clearly, for every $I \subseteq S$ there is a \textbf{canonical retraction}
$\rho_I \in {\rm End(W)}$ of $W$ onto the standard parabolic subgroup $W_I$,
defined by $\rho_I(s):=s$ for all $s \in I$ and $\rho_I(t):=1$ for all $t\in S \setminus I$. It is also obvious that for any other subset
$J \subseteq S$, the retractions $\rho_I$ and $\rho_J$ commute.

\begin{prop} \label{prop:f_o_in_Cox-cd} If $W$ is an even Coxeter group of finite rank then every finite order element is conjugacy distinguished in $W$.
\end{prop}

\begin{proof}
Let $x \in W$ be an element of finite order.
By Lemma \ref{lem:PcFinite} we can assume that $x \in W_I$ for some $I \subseteq S$ such that $|W_I|<\infty$.
Consider any $y \in W \setminus x^W$. If $y$ has infinite order then, since W is residually finite (as any finitely generated linear group -- see \cite{Malcev}),
there is a finite group $F$ and a homomorphism $\alpha:W \to F$ such that
the order of $\alpha(y)$ in $F$ is greater than the order of $x$ in $W$. Clearly this implies that $\alpha(y) \notin \alpha(x)^F$.

Thus we can suppose that $y$ has finite order, and so, by Lemma \ref{lem:PcFinite}, $y$ is conjugate in $W$ to an element $W_J$ for some
$J \subseteq S$ with $|W_J|<\infty$. Without loss of generality, we can replace $y$ with its conjugate to assume that $y \in W_J$.
In view of Remark \ref{rem:inter_comm_retr}, we see that
$W_I\cap W_J =\rho_{I\cap J}(W)=W_{I \cap J}$. So, since $y \notin x^W$, Lemma \ref{lem:retr_crit} tells us that that either
$\rho_I(y) \notin x^{W_I}$ in $W_I$, or $\rho_J(x) \notin y^{W_J}$ in $W_J$,
or $\rho_{I \cap J}(y) \notin \left(\rho_{I\cap J}(x)\right)^{W_{I\cap J}}$ in $W_{I\cap J}$. Let us assume that $\rho_I(y) \notin x^{W_I}$ in $W_I$,
as the other two cases are similar.
As $x \in W_I$, we have $x=\rho_I(x)$, thus $\rho_I:W \to W_I$ is the homomorphism from $W$ to a finite group $W_I$, distinguishing the conjugacy
classes of the images of $x$ and $y$.
Hence $x \in W$ is conjugacy distinguished.
\end{proof}

Recall that an element $x$ of a Coxeter group $W$ is called \textbf{essential} if $\pc(x)=W$.  {Remark that this notion refers to a specific Coxeter generating set of $W$.}

\begin{lem}
\label{lem:essential-power} Suppose that $W$ is an infinite irreducible Coxeter group of
finite rank and $x \in W$ is an essential element. Then for every $m \in \N$, $x^m$ is also an essential element of $W$.
\end{lem}

\begin{proof} Since every element of finite order is contained in a finite parabolic subgroup (Lemma \ref{lem:PcFinite}) and $|W|=\infty$, we
see that, being an essential element, $x$ must have infinite order.
Let $P:=\pc(x^m) \leqslant W$.  {According to Lemma \ref{lem:ParabHered}}, $x^{-1}Px \cap P$ is a parabolic subgroup containing $\gen{x^m}$, hence
$x^{-1}Px \cap P=P$ by the minimality of $P$, implying that $P \subseteq x^{-1}Px$. Similarly, $P \subseteq xPx^{-1}$, hence $P=x^{-1}Px$, i.e.,
$x$ belongs to the normalizer $N_W(P)$ of $P$ in $W$. Since $x$ is essential in $W$, it follows that $\pc(N_W(P))=W$. By a result of Krammer
\cite[Lemma 6.8.1]{Krammer}, the latter implies that either $|P|<\infty$ or $P=W$. But $P$ cannot be finite since $x^m \in P$ has infinite order, therefore
$P=W$, thus $x^m$ is essential in $W$.
\end{proof}

A Coxeter group is said to be \textbf{affine} if it isomorphic to a Euclidean reflection group. The following statement was proved by Krammer in \cite[Lemma 6.3.10]{Krammer}:

\begin{lem}\label{lem:Krammer} Let $W$ be an infinite, irreducible and non-affine Coxeter group of finite rank. If $x \in W$ is an essential element then $\gen{x}$
has finite index in the centralizer
$C_W(x)$ of $x$ in $W$.
\end{lem}

\begin{lem}\label{lem:fi_retr-cs} Let $G$ be a group with subgroups $A,H \leqslant G$ such that $|G:H|<\infty$ and $A$ is a retract of $G$. If $H$ is hereditarily
conjugacy separable then $A \cap H$
is hereditarily conjugacy separable.
\end{lem}

\begin{proof} Consider any finite index subgroup $A'$ of $A \cap H$. Let $\rho:G \to A$ be a retraction of $G$ onto $A$ and let $K:=\rho^{-1}(A') \cap H \leqslant G$.
Observe that  $|A:(A\cap H)| \le |G:H|<\infty$, therefore $|A:A'|<\infty$ and
$|H:K|\le |G:\rho^{-1}(A')|=|A:A'|<\infty$. Moreover, it is easy to see that  $A' \subseteq K$ and $\rho(K)\subseteq A'$, implying that the restriction of $\rho$ to $K$
is a retraction of $K$ onto $A'$.

Since $H$ is hereditarily conjugacy separable and $|H:K|<\infty$, $K$ is conjugacy separable, i.e., $a^K$ is separable in $K$ for each $a \in A'$. Note that $a^K \cap A'=a^{A'}$
(indeed, if $f^{-1} a f \in a^K \cap A'$ for some $f \in K$ then $f^{-1}af=\rho(f^{-1}af)={f'}^{-1}af' \in a^{A'}$, where $f':=\rho(f) \in A'$), hence $a^{A'}$ is closed in the
subspace topology on $A'$, induced by the profinite topology of $K$. In view of Remark \ref{rem:subspace_top} we see that $a^{A'}$ is
separable in $A'$. Thus any finite index subgroup $A' \leqslant A\cap H$ is conjugacy separable, i.e., $A\cap H$ is hereditarily conjugacy separable.
\end{proof}

\begin{lem} \label{lem:cox_cyc_sbgp_sep} Any {amenable} subgroup of a finitely generated Coxeter group   is closed in the profinite topology.
\end{lem}

\begin{proof}
{Coxeter groups are CAT($0$) groups by \cite[Th.~12.3.3]{Davis}. Therefore every amenable subgroup is virtually abelian by \cite[Cor.~B]{AdamsBallmann}.}

By a theorem of Haglund and Wise \cite[Cor. 1.3]{Haglund-Wise_Cox}, any finitely generated Coxeter group $W$ contains a finite index subgroup $G$ such that $G$ is a subgroup of some right
angled Coxeter group $R$ of finite rank.
The standard geometric representation of $R$ (see \cite[5.3]{Humphreys}) is a faithful representation (\cite[Cor. 5.4]{Humphreys}) by matrices with
integer coefficients. It follows that
$G$ is a finitely generated subgroup of ${\rm GL}_n(\Z)$ for some $n \in \N$. Segal proved (see \cite[Thm. 5, p. 61]{Segal}) that every solvable subgroup of
${\rm GL}_n(\Z)$ is closed in the profinite topology
of that group. Hence every solvable subgroup of $G$ is separable in $G$ (by Remark \ref{rem:subspace_top}).

So, let $L$ be a virtually solvable subgroup of $W$. Since $|W:G|<\infty$ we can find a solvable subgroup
$M \leqslant L \cap G$ and $f_1,\dots,f_k \in L$ such that $L=\bigsqcup_{i=1}^kf_i M$. By the above, $M$ is
separable in $G$, therefore, according to  Remark \ref{rem:subspace_top}, it is also separable in $W$.
Consequently, $L$ is closed in $\mathcal{PT}(W)$ as a finite union of closed sets.
\end{proof}

We will now apply the Chagas-Zalesskii criterion \cite{C-Z_bianchi} to obtain

\begin{lem}\label{lem:ess-cd} Let $W$ be an infinite non-affine irreducible Coxeter group {of finite rank}. If $W$ has a finite
index  {hereditarily conjugacy} separable subgroup $H$ then every essential element in $W$ is conjugacy distinguished.
\end{lem}

\begin{proof} Evidently we can assume that $H$ is normal in $W$. Consider any
essential element $x \in W$. Set $m :=|W:H| \in \N$, then $x^m$  is also essential in $W$ by Lemma \ref{lem:essential-power}. Therefore,
according to Lemma \ref{lem:Krammer}, the centralizer $C_W(x^m)$ is virtually cyclic and hence it is conjugacy separable (cf. \cite{Rem,Form}).
Also, every subgroup of $C_W(x^m)$ is virtually cyclic, and so it
is separable in $W$ by Lemma \ref{lem:cox_cyc_sbgp_sep}. Therefore we can apply Proposition \ref{prop:C-Z-crit} to conclude that $x^W$ is separable in $W$.
\end{proof}

{The proof of the next statement combines the criteria from Proposition \ref{prop:C-Z-crit} and Lem\-ma~\ref{lem:retr_crit}.}
\begin{prop}\label{prop:even_Cox_with_hcs_sbgp-cs} Suppose that $W$ is an even Coxeter group of finite rank that contains a finite
index normal subgroup $H \lhd W$ such that $H$ is hereditarily conjugacy
separable. Then $W$ is conjugacy separable.
\end{prop}

\begin{proof} The proof will proceed by induction on the rank $\rank(W)=|S|$, where $S$ is a fixed Coxeter generating set of $W$.
If $\rank(W) \le 1$ then $W$ is finite and the claim trivially holds. So suppose that $\rank(W)>1$ and the claim has already
been established for all even Coxeter groups of rank less than $\rank(W)$.
If $W$ is finite then there is nothing to prove; if $W$ is affine, then it is virtually abelian and so it is
conjugacy separable (as any virtually polycyclic group -- see \cite{Rem,Form}).

If $W$ is not irreducible, then $W=W_I \times W_J$ for some $I,J \subsetneqq S$ such that $S=I \sqcup J$. Note that $W_I$ is an even
Coxeter group with $\rank(W_I)=|I|<|S|=\rank(W)$ and
$W_I \cap H$ is a hereditarily conjugacy separable subgroup of finite index in $W_I$ (by Lemma \ref{lem:fi_retr-cs}). By the
induction hypothesis, $W_I$ is conjugacy separable; similarly, $W_J$ is
conjugacy separable. It is easy to check that the direct product of two conjugacy separable groups is conjugacy separable, hence $W=W_I \times W_J$ is conjugacy separable.

Therefore we can further assume that $W$ is infinite, non-affine and irreducible.
Take an arbitrary element $x \in W$.
If $x$ is essential in $W$ then $x$ is conjugacy distinguished by Lemma~\ref{lem:ess-cd}.
Thus we can further assume that  $x$ is not an essential element of $W$.
In this case, after replacing $x$ with its conjugate, we can suppose that $x \in W_I$ for some $I \subsetneqq S$.
Choose any $y \in W \setminus x^W$. If $y$ is an essential element of $W$, then $y^W$ is separable in $W$ by Lemma \ref{lem:ess-cd}. Since $x \notin y^W$,
there is a finite group $F$ and a
homomorphism $\alpha:W \to F$ such that $\alpha(x) \notin \alpha(y)^F$, which is equivalent to  $\alpha(y) \notin \alpha(x)^F$ in $F$.
The latter means that $y$ does not belong to the closure of $x^W$ in $\mathcal{PT}(W)$.

So, we can suppose that $y$ is not essential in $W$, which, without loss of generality, allows us to assume that $y\in W_J$ for some $J \subsetneqq S$. Since $y \notin x^W$,
using Lemma \ref{lem:retr_crit} we see that either $\rho_I(y) \notin x^{W_I}$ in $W_I$, or $\rho_J(x) \notin y^{W_J}$ in $W_J$,
or $\rho_{I \cap J}(y) \notin \left(\rho_{I\cap J}(x)\right)^{W_{I\cap J}}$ in $W_{I\cap J}$. Let us focus on the case when $\rho_I(y) \notin x^{W_I}$ in $W_I$ as
the other two cases are similar.

Since $\rank(W_I)=|I|<|S|=\rank(W)$, the induction hypothesis holds (in view of Lemma \ref{lem:fi_retr-cs}), and so $W_I$ is conjugacy separable. Therefore, there is
a finite group $F$ and a homomorphism $\alpha:W_I \to F$ such that $\alpha(\rho_I(y)) \notin \alpha(x)^F$ in $F$. Thus the homomorphism $\alpha\circ\rho_I:W \to F$ separates
the the image of $y$ from the conjugacy class of the image of $x$ in $F$. Since such a homomorphism has been found for an arbitrary $y \in W \setminus x^W$, we are able
to conclude that $x^W$ is separable in $W$, which finishes the proof of the proposition.
\end{proof}

\begin{proof}[Proof of Theorem \ref{thm:some_even_Cox-cs}]  By Corollary 1.5 from \cite{Caprace-Muhlherr}, any Coxeter group whose Coxeter diagram does not contain irreducible affine subdiagrams
of rank at least $3$ acts cocompactly on the associated
Niblo--Reeves cube complex (see \cite{Niblo-Reeves}). Since the only even irreducible affine Coxeter diagram of rank $\ge 3$ is $\widetilde{B}_2$
(according to the classification of all irreducible affine Coxeter groups --  see, {for example}, \cite[Appendix C]{Davis}),
our assumptions imply that $W$ acts cocompactly on its Niblo--Reeves cubing.

As discussed in the introduction,
the results of Haglund and Wise from
\cite{Haglund-Wise_Special,Haglund-Wise_Cox} combined with the main theorem of \cite{Min-RAAG} imply that
every Coxeter group, whose action on the associated Niblo--Reeves cube complex is cocompact, has a hereditarily conjugacy separable subgroup of finite index.
Therefore, $W$ satisfies all the assumptions of Proposition~\ref{prop:even_Cox_with_hcs_sbgp-cs}, allowing us to conclude that it is conjugacy separable.
\end{proof}

\section{Reflection subgroups of Coxeter groups}

A \textbf{reflection subgroup} of $W$ is defined as a subgroup of $W$ generated by reflections. For example each parabolic subgroup is a reflection subgroup.
It is a general fact that a reflection subgroup is itself a Coxeter group. We shall need the following more precise version of this fact.

\begin{prop}\label{prop:ReflSubgroup}
Let   $G \leqslant W$ be a reflection {subgroup}.

\begin{enumerate}[(i)]
\item There is a set of reflections $R \subset G$ such that $(G, R)$ is a Coxeter system.

\item Let $\Gamma_{(W, S)}$ (resp. $\Gamma_{(G, R)}$) be the Cayley graph of $(W, S)$ (resp. $(G, R)$). Let $\Gamma$ be the quotient graph of $\Gamma_{(W, S)}$
obtained by collapsing each edge stabilised by a reflection which does not belong to $G$. Then $\Gamma$ is $G$-equivariantly isomorphic to $\Gamma_{(G, R)}$.

\end{enumerate}
\end{prop}

\begin{proof}
See \cite{DeoRefl} or \cite{Dyer}.
\end{proof}

We emphasize that, as opposed to the case of parabolic subgroups, it is not true in general that an intersection of reflection subgroups is itself a reflection subgroup.
Indeed, consider the infinite dihedral group $W$. It has two conjugacy classes of reflections, each generating a reflection subgroup which is of finite index in $W$.
The intersection of these two subgroups is torsion-free and of finite index in $W$, hence it is not a reflection subgroup.

The following strengthening of Lemma~\ref{lem:CoxeterElement} shows however that for some specific elements in $W$, there is a unique minimal reflection subgroup containing them.

\begin{lem}\label{lem:CoxeterElement:Refl}
Let $J = \{s_1, \dots, s_n\} \subseteq S$ and denote by $w = s_1 s_2 \dots s_n$ the product of all elements of $J$ (ordered arbitrarily).

Then every reflection subgroup of $W$
containing $w$ also contains $J$. In particular $W_J$ is the unique minimal reflection subgroup of $W$ containing $w$.
\end{lem}

This will be deduced from the following.

\begin{lem}\label{lem:SubRefl}
Let  $w  = s_1 \dots s_n \in W$  be a reduced word. Set $r_i = s_1 s_2 \dots s_{i-1}s_i s_{i-1} \dots s_1$ for all $i = 1, \dots, n$.

For any reflection subgroup $G \leqslant W$ containing $w$,
 we have $w \in \la G \cap \{r_1, \dots, r_n\} \ra$.
\end{lem}

\begin{proof}
Let $\mathcal W_i$ be the wall fixed by $r_i$ in the Cayley graph  $\Gamma_{(W, S)}$ of $(W, S)$. Since $w  = s_1 \dots s_n \in W$ is reduced, it follows that
$\mathcal W_1, \dots, \mathcal W_n$ are exactly the walls separating $1$ from $w$ in $\Gamma_{(W, S)}$. These walls are successively crossed by a minimal
path $\gamma$ from $1$ to $w$.

Let now $\varphi : \Gamma_{(W, S)} \to \Gamma$ be the $G$-equivariant map to the Cayley graph of $G$ provided by Proposition~\ref{prop:ReflSubgroup}.
Then $\varphi(\gamma)$ is a path joining $1$ to $w$ in the Cayley graph $\Gamma$. Let $\mathcal W'_1, \dots , \mathcal W'_m$ be the walls of $\Gamma$
successively crossed by $\varphi(\gamma)$, and let $r'_i$ be the reflection fixing $\mathcal W'_i$. Thus we have
$w = r'_m \dots r'_1$.

Proposition~\ref{prop:ReflSubgroup}(ii) implies that
$$\{r'_1, \dots, r'_m\} =  \{r_1, \dots, r_n\} \cap G.$$
The desired result follows.
\end{proof}

\begin{proof}[Proof of Lemma~\ref{lem:CoxeterElement:Refl}]
We proceed by induction on $n = | J |$, the base case $n =1$ being trivial.

Let $G \leqslant W$ be a reflection subgroup containing $w = s_1 s_2 \dots s_n$. For each $i$, let $r_i = s_1 s_2 \dots s_{i-1}s_i s_{i-1} \dots s_1$.

We claim that $r_n \in G$. If this were not the case, then Lemma~\ref{lem:SubRefl} would imply that $w \in \la r_1, \dots, r_{n-1} \ra = \la s_1, \dots, s_{n-1}\ra$,
contradicting Lemma~\ref{lem:CoxeterElement}.

The claim implies that $s_1 \dots s_{n-1} = r_n w$ is contained in $G$. By induction, this implies that $G$ contains $\{s_1, \dots, s_{n-1}\}$.
Since $G$ also contains {$r_n $},
it follows that $G$ contains $s_n$, whence $J \subset G$ as desired.
\end{proof}

\section{Automorphisms preserving parabolic subgroups up to conjugacy}\label{sec:ParabPreserv}

Let $W$ be a finitely generated Coxeter group.
Two Coxeter generating sets $S_1$, $S_2$ for $W$ are called \textbf{reflection-compatible} if each element of $S_1$ is
conjugate to an element of $S_2$. They are called \textbf{angle-compatible} if they are reflection-compatible and if,
moreover, for each spherical pair $\{s, t\} \subseteq S_1$, there is $w \in W$ such that $\{wsw^{-1}, wtw^{-1}\} \subseteq S_2$.}
Furthermore, we say that $S_1$ and $S_2$ are \textbf{parabolic-compatible} if for every $J_1 \subseteq S_1$,
there is some $J_2 \subseteq S_2$ such that the subgroup $W_{J_1}$ is conjugate to $W_{J_2}$.

It is important to remark that reflection-compatibility, angle-compatibility and parabolic-compatibility are equivalence
relations on the collection of all Coxeter generating sets. For the first two relations, see \cite[Appendix~A]{TwistRigid};
for parabolic-compatibility, this follows from Lemma~\ref{lem:ParabCompatible} below.

The following basic observation is useful.

\begin{lem}\label{lem:Card}
Let $W$ be a finitely generated Coxeter group. Any two Coxeter generating sets which are reflection-compatible have the same cardinality.
\end{lem}

\begin{proof}
Follows from basic considerations using root systems. {The desired statement boils down to the property that any two bases of a vector space have the same cardinality.}
\end{proof}


{Remark that two Coxeter generating sets that are not reflection-compatible need not have the same cardinality.
For example, the dihedral group of order $12$ is isomorphic to the direct product of the dihedral group of order $6$ with the cyclic group of order $2$.}


{
Clearly, the relation of parabolic-compatibility is much stronger than reflection-compa\-tibility among Coxeter generating sets for $W$. For example, if $W$ is a
\textbf{free} Coxeter group, i.e., a free product of groups of order~$2$, then any two Coxeter generating sets are reflection compatible (because any involution
in $W$ is a reflection in that case), but if the rank of $W$ is at least~$3$, it is easy to find automorphisms that do not map every parabolic subgroup to a parabolic subgroup.

The following lemma shows however that reflection-compatibility is sufficient to ensure the compatibility of all \emph{spherical} parabolic subgroups.

\begin{lem}\label{lem:Autom}
Let  $S, S'$ be reflection-compatible Coxeter generating sets for a Coxeter group $W$.
Then for each spherical subset $J \subseteq S$, there is a subset $J' \subseteq S'$ with $|J| = |J'|$ such that  $W_J$ and $W_{J'}$ are conjugate.
\end{lem}

\begin{proof}
Let $n = |J|$ and $P = \Pc(W_J)$ be the parabolic closure of $W_J$ with respect to the Coxeter generating set $S'$. Then $P$ is a finite parabolic subgroup of rank~$k \leq n$ with respect to $S'$, by Lemma~\ref{lem:PcFinite}.
The lemma  also implies that the parabolic closure $Q$ of $P$ with respect to $S$  is a parabolic subgroup of rank~$k ' \leq k \leq n$ with respect to $S$.
Since $W_J \leqslant P$,
we have $W_J \leqslant Q$. Since $W_J$ is of rank $n$ and $Q$ is of rank~$k' \leq n$, it follows that $W_J = Q$
and  $k' = n$.
In particular $W_J = P = Q$ and $k = k' = n$ so that $W_J$ is a finite parabolic of rank $n$ with respect to $S'$, as desired.
\end{proof}

}

We shall also need the following technical fact, showing that the various notions of compatibility are appropriately inherited by parabolic subgroups.

\begin{lem}\label{lem:CompatHered}
Let $W$ be a finitely generated Coxeter group, and $S_1, S_2$ be two Coxeter generating sets.  Let $J_1 \subseteq S_1$ and $J_2 \subseteq S_2$ be
such that $W_{J_1} = W_{J_2}$.

If $S_1$ and $S_2$ are reflection-compatible (resp. angle-compatible, parabolic-compatible), then so are  $J_1$ and $J_2$ as Coxeter
generating sets for the Coxeter subgroup $W_{J_1} = W_{J_2}$.
\end{lem}

\begin{proof}
{We start with the following observation, which is a special case of Lemma~\ref{lem:ParabHered}:
  if $P, Q$ are parabolic subgroups of a Coxeter group $W$ and if $P$ is contained in $Q$,
then $P$ is also parabolic as a subgroup of the Coxeter group $Q$.

The above observation is already enough to draw the desired conclusion for reflection-compatibility and parabolic-compatibility.

Assume, now, that $S_1$ and $S_2$ are angle-compatible. This implies that in the  Cayley graph of $(W, S_2)$,
viewed as a chamber system, every spherical pair $\{s, t\} \subseteq S_1$ fixes two walls that contain two panels $\sigma, \tau$ of a
common chamber, say $c$. Given any residue $R$ stabilised by $P = \la s, t\ra$, consider the combinatorial projection of $c$, $\sigma$
and $\tau$ onto $R$, and call them $c'$, $\sigma'$ and $\tau'$. By properties of the projection (see \cite[2.30--2.32]{TitsLN}), $\sigma'$
and $\tau'$ must be panels stabilised by $s$ and $t$ respectively; moreover, they are both panels of the chamber $c'$.

We now fix a residue $R_0$,  whose stabiliser is $W_{J_1} = W_{J_2}$. If  the pair $\{s, t\}$ is contained in $J_1$,
then, by Lemma~\ref{lem:PcFinite}, we can find a rank two residue $R$ stabilised by $P$ within  $R_0$. By the preceding paragraph,
there is a chamber of $R$ containing two panels respectively stabilised by $s$ and $t$. Hence the element of $W_{J_1} = W_{J_2}$, that
maps this chamber to the base chamber $1$, conjugates the pair $\{s, t\} \subset J_1$ to a pair contained in $J_2$. This shows
that $J_1$ and $J_2$ are angle-compatible  as Coxeter
generating sets for the Coxeter subgroup $W_{J_1} = W_{J_2}$.
}
\end{proof}

\begin{lem}\label{lem:ParabCompatible}
Let $S_1, S_2$ be two Coxeter generating sets   for $W$.   If for every $J_1 \subseteq S_1$, there is some $J_2 \subseteq S_2$ such that the subgroup
$W_{J_1}$ is conjugate to $W_{J_2}$ {in $W$}, then for every $J_2 \subseteq S_2$, there is some $J_1 \subseteq S_1$ such that the subgroup
$W_{J_2}$ is conjugate to $W_{J_1}$  {in $W$}.
\end{lem}

\begin{proof}
The hypothesis implies that $S_1$ and $S_2$ are reflection-compatible.
By Lemma~\ref{lem:Card}, we have $|S_1| = |S_2|$.

Note  that if $W_{J_1}$ is conjugate to
$W_{J_2}$ for some $J_1 \subseteq S_1$ and $J_2 \subseteq S_2$, then,
after replacing $S_2$ by some conjugate, we can assume that   $W_{J_1}=W_{J_2}$.
The assumptions together with Lemma~\ref{lem:CompatHered} imply that $J_1$ and $J_2$ are reflection-compatible within the Coxeter group  $W_{J_1}=W_{J_2}$.
Therefore  $J_1$ and $J_2$ have the same cardinality by Lemma~\ref{lem:Card}.

\medskip
{Assume,  {at first,} that $S_1$ is irreducible.}

Let $J_2 \subseteq S_2$. We need to find  some $J_1 \subseteq S_1$ such that  $W_{J_1}$ is conjugate
to $W_{J_2}$. Since any intersection of parabolic subgroups is again parabolic, it suffices to consider
the case when $J_2$ is a maximal proper subset of $S_2$. Suppose for a contradiction that for some such
maximal subset $J_2 \subset S_2$, the group $W_{J_2}$ is not parabolic with respect to $S_1$. We know from
the previous paragraph that for each proper maximal subset $I \subset S_1$, there is a proper maximal subset
$J \subseteq S_2$ such that  $W_{I}$ is conjugate to $W_{J}$. There are exactly $n$ proper maximal subsets
of $S_1$ (resp. $S_2$), where $n = |S_1| = |S_2|$. By assumption, $W_{J_2}$ is not conjugate to any $W_I$
with $I \subset S_1$. Therefore, there must be  two distinct proper maximal subsets $I, I' \subset S_1$ such
that $W_{I}$ and $W_{I'}$ are conjugate to the same group $W_J$ for some $J \subset S_2$. In particular $W_I$
and $W_{I'}$ are conjugate. Since $S_1$ is irreducible, this implies by  \cite{Deodhar} or \cite[\S 3.1]{Krammer}
that $S_1$ is spherical; in other words $W$ is finite. In particular $W_{J_2}$ is a finite subgroup generated by $n-1$
reflections. Lemma~\ref{lem:PcFinite} thus implies that there is some $J_1 \subset S_1$ of cardinality at most $n-1$
such that $W_{J_1}$ contains $wW_{J_2}w\inv$ for some $w \in W$. By hypothesis there is $I_2 \subseteq S_2$ such that
$W_{J_1} = g W_{I_2} g\inv$ for some $g \in W$. As observed above the sets $J_1$ and $I_2$ have the same cardinality,
which is at most $n-1$. Thus $W_{J_2}$ is conjugate to a subgroup of $W_{I_2}$. Since $W_{J_2}$ is a parabolic subgroup
of rank $n-1$, it cannot be contained in a parabolic subgroup of any smaller rank. Moreover, two parabolic subgroups of
the same rank must coincide if one is contained in the other. Thus $W_{J_2}$ must be conjugate to $W_{I_2}$, and hence
also to $W_{J_1}$. This is a contradiction.

\medskip
{Assume, finally, that $S_1$ is reducible. We shall finish the proof by induction on the number of irreducible components of $S_1$.
The base case of the induction, i.e., when $S_1$ is irreducible, has already been established.

Suppose that  $S_1$ is a disjoint union $S_1 = I_1 \cup I'_1$ of two non-empty commuting subsets. Then $W_{I_1}$ and $W_{I'_1}$ are normal in $W$.
By hypothesis, they are also standard parabolic subgroups with respect to $S_2$. Set $I_2 = S_2 \cap W_{I_1}$ and  $I'_2 = S_2 \cap W_{I'_1}$.
Since $W =  W_{I_1} \times W_{I'_1}$, we have $S_2 = I_2 \cup I'_2$ and the two sets $I_2$ and $ I'_2$ commute.

By Lemma~\ref{lem:CompatHered}, the Coxeter generating sets $I_1$ and $I_2$ (resp.  $I'_1$ and $I'_2$) are parabolic-compatible in the Coxeter subgroup $W_{I_1}$ (resp. $W_{I_2}$). By induction, for any subset $J_2 \subset I_2$ (resp. $J'_2 \subset I'_2$), we find some $w \in W_{I_1}$ (resp. $w' \in W_{I'_1}$) and some $J_1\subset I_1$ (resp. $J'_1 \subset I'_1$) such that $wW_{J_2}w\inv  =W_{J_1}$ (resp.$ w'W_{J'_2}(w')\inv  =W_{J'_1}$). Since $w$ (resp. $w'$) commutes with $W_{I'_1}$ (resp. $W_{I_1}$), it follows that $ww'$ conjugates $W_{J_2 \cup J'_2}$ onto $W_{J_1 \cup J'_1}$. This finishes the proof.}
\end{proof}

The goal of this section is to establish the following fact, {which will later be used to prove Theorem~\ref{thm:CoxInner} from the Introduction}.

\begin{prop}\label{prop:MainTech}
Let $W$ be a finitely generated Coxeter group, and $S_1, S_2$ be two Coxeter generating sets.

If $S_1, S_2$ are angle-compatible and parabolic-compatible, then there is some inner automorphism $\alpha \in \Inn(W)$ such that $\alpha(S_1) = S_2$.
\end{prop}

The condition that $S_1 $ and $S_2$ are parabolic-compatible is not sufficient on its own to guarantee that they are conjugate; examples
illustrating this matter of fact may be found amongst finite dihedral groups.

{
A subset of a Coxeter generating set is called \textbf{$2$-spherical} if every pair of elements in it is spherical.  We shall need the following elementary fact.

\begin{lem}\label{lem:Induct}
Let $W$ be a Coxeter group with Coxeter generating set $S$. Let $J \subseteq S$ be a subset that is
irreducible, but not $2$-spherical. If $|J|>2$, then   there is some $s \in J$ such that $J \setminus \{s\}$ is still irreducible and non-$2$-spherical.
\end{lem}

\begin{proof}
Start with a pair of elements $I_0 \subset J$ that generates an infinite dihedral group.
Since $J$ is irreducible, the pair $I_0$ must be contained in a triple $I_1 \subset J$
which is still an irreducible subset. Proceeding inductively, we construct a chain
$I_0 \subsetneq I_1 \subsetneq \dots$ such that $|I_n| = n+2$ and each $I_i$ is irreducible and not $2$-spherical. The result follows.
\end{proof}
}

\begin{proof}[Proof of Proposition~\ref{prop:MainTech}]

By Lemma~\ref{lem:Card}, we have $n = |S_1| = |S_2|$. We proceed by induction on $n$, the base case $n=1$ being trivial. 

Given an irreducible component $J_1 \subseteq S_1$, there is $J_2 \subseteq S_2$ such that $W_{J_1}$ is conjugate to $W_{J_2}$.
Since $J_1$ is an irreducible component, the parabolic group $W_{J_1}$ is normal in $W$, and we infer that $W_{J_1} = W_{J_2}$.
Moreover $J_1$ and $J_2$ are angle-compatible and parabolic-compatible by Lemma~\ref{lem:CompatHered}.
 By induction, we may therefore assume henceforth that $S_1$ and $S_2$ are both irreducible.

Assume that $S_1$ is {$2$-spherical}.
It follows from \cite[Prop.~11.7]{CM07}
 that a Coxeter generating set for $W$ is conjugate to $S_1$ if and only if it is angle-compatible with $S_1$. Thus we are done in this case.

We assume henceforth that  {neither $S_1$ nor $S_2$ are $2$-spherical}.
We can moreover assume that $n  > 2$, since otherwise $W$ would be
infinite dihedral, in which case the desired result is trivial since any two Coxeter generating sets are conjugate in that case.

Since $S_1$ is irreducible but not $2$-spherical,  it contains an element $s_0$ such that $S'_1 = S_1 \setminus \{s_0\}$ is still
irreducible and not $2$-spherical by Lemma~\ref{lem:Induct}.

Since $W_{S'_1}$ is conjugate to $W_{S'_2}$ for some $S'_2 \subseteq S_2$, we may assume, after replacing $S_2$ by a
conjugate,  that $W_{S'_1} = W_{S'_2}$. By Lemma~\ref{lem:CompatHered}, the sets $S'_1$ and $S'_2$ are angle-compatible
and parabolic-compatible in the Coxeter group $W_{S'_1}$. Therefore, the induction hypothesis implies that $S'_1$ and
$S'_2$ are conjugate. After replacing $S_2$ by a conjugate, we may thus assume that $S'_1 = S'_2$. We set $S' = S'_1 = S'_2$
and denote by $s'_0$ the unique element of $S_2 \setminus S'$.

Now we distinguish two main cases.

Assume first that $S'$ has exactly two elements, say $S' = \{s_1, s_2\}$. In that case $n = 3$ and we will conclude by
analyzing successively the few possible situations as follows.

If no pair in $S_1$ is spherical, then the Cayley graph of $W$ is the trivalent tree $T$. The hypotheses imply that $s_i s'_0$
is a translation of length $2$ for $i = 1, 2$.
{Moreover, so is the product $s_1 s_2$. It follows that the three reflections $s_1, s_2, s'_0$ fix three edges of $T$ that are mutually at distance
$1$ from one another. Thus these three edges have a common vertex, which must \emph{a fortiori} be the common vertex between the edges fixed by
$s_1$ and $s_2$. Therefore  $s_0 = s'_0$, and we are done.

If $\{s_0, s_1\}$ is spherical and $\{s_0, s_2\}$ is not, then  $W$ is the free product of a finite dihedral group $\la s_0, s_1\ra$ and a
cyclic group of order $2$ generated by $s_2$.  {This decomposition of $W$ as a free product must  also be visible with respect to the Coxeter generating set $S_2$. Indeed,  the pair $\{s'_0, s_2\}$ is non-spherical
({because $\la s_2 \ra$ is a free factor of $W$ of order $2$}). Therefore the pair $\{s'_0 , s_1\}$ must be spherical, {by angle-compatibility of $S_1$ and $S_2$}.
}


We next remark that $\la s_0 , s_1 \ra$ is  {the unique maximal finite subgroup of $W$ containing $s_1$}.  Since $\la s'_0, s_1\ra$ is
such a finite subgroup, we must have $s'_0 \in \la s_0, s_1\ra$. Now, the property that $s'_0 s_2$ is a
translation of length $2$ in the Cayley graph of $(W, S_1)$ forces $s'_0 = s_0$, as desired.

If  $\{s_0, s_1\}$ and $\{s_0, s_2\}$ are both spherical, then $W$ splits as the amalgamated product $W=\langle s_1,s_0 \rangle*_{\langle s_0\rangle} \langle s_0,s_2 \rangle$.
We claim that  $W$ contains a unique non-trivial element $x=s_0$ such that both of the subgroups $\langle x,s_1\rangle$ and $\langle x,s_2\rangle$ are finite.

{Since the pairs   $\{s_1,s_0'\}$ and $\{s_2,s_0'\}$ are both spherical (otherwise $W$ would have $\la s_1\ra$ or $\la s_2 \ra$ as a free factor of order two, which is impossible since $\{s_0, s_1\}$ and $\{s_0, s_2\}$ are both spherical), that} claim readily implies
that $s_0 = s'_0$, which concludes the proof in the special case at hand.


{The claim can be established as follows.} Since $W$ is an amalgamated product,
it acts on the associated Bass-Serre tree $T$. Suppose that $x \in W$ is an element such that both of the subgroups $\langle x,s_1\rangle$ and
$\langle x,s_2\rangle$ are finite. A finite group acting on  a tree always fixes some vertex (see \cite[Example~I.6.3.1]{Serre}), hence
there are vertices $u_1,u_2$ of $T$ such that $u_1 \in \Fix(x)\cap \Fix(s_1)$ and $u_2 \in \Fix(x) \cap \Fix(s_2)$, where $\Fix(x)$
denotes the set of vertices of $T$ fixed by $x$. Note that $\Fix(s_1)$ and $\Fix(s_2)$ are two convex subsets of $T$ with empty intersection,
because vertex stabilisers (for the action of $W$ on $T$) are finite and the pair $\{s_1,s_2\}$ is not spherical by the assumptions.
Therefore there is a unique edge $e$ with $e_- \in \Fix(s_1)$ and $e_+ \in \Fix(s_2)$. The stabiliser of $e$ in $W$ is the subgroup $\langle s_0 \rangle$ and
any arc in $T$ connecting a vertex of $\Fix(s_1)$ with a vertex $\Fix(s_2)$ must pass through $e$. Since $x$ fixes one of such arcs $[u_1,u_2]$, we deduce that
$x$ fixes $e$. As $\langle s_0 \rangle$ contains only one non-trivial element, we can conclude that $x=s_0$, thereby proving the claim.}

Assume now that $S'$ has more than two elements.
{By Lemma~\ref{lem:Induct}, there is some $s_1 \in S'$ such that $S'' = S_1 \setminus \{s_1\}$ is still irreducible and  not $2$-spherical.}
Since $S'$ is irreducible,
it follows that $(S'')^\perp \subseteq \{s_0\}$, where for a subset $J \subseteq S$,
$J ^\perp$ denotes the set of those $s \in S \setminus J$ commuting with $J$  {in $W$}.

If $(S'')^\perp = \{s_0\}$, then the centraliser of $S''$ in $W$ is $\la s_0 \ra$ by Lemma~\ref{lem:Centra}(i).
Applying the same lemma with respect to the Coxeter generating set $S_2$ then yields that the centraliser
of $S''$ is $\la s'_0 \ra$. Therefore  $s_0 = s'_0$   and we are done in this case.

If $(S'')^\perp =\varnothing$, then $W_{S'}$ and $W_{S'' \cup \{s_0\}}$ are the only two proper
parabolic subgroups of $W$ (with respect to the Coxeter generating set $S_1$) containing $W_{S''}$
properly, by Lemma~\ref{lem:Centra}(iii). Since $S_1$ and $S_2$ are parabolic-compatible, we infer
that $W_{S'' \cup \{s_0\}} = W_{S'' \cup \{s'_0\}}$.

By Lemma~\ref{lem:CompatHered}, the sets $S'' \cup \{s_0\}$ and $S'' \cup \{s'_0\}$ are angle- and parabolic-compatible.
Thus by induction there is some $w \in W_{S'' \cup \{s_0\}}$ such that $wS'' w\inv \cup\{w s'_0 w\inv\} = S'' \cup \{s_0\}$.
{Since $S''$ is irreducible non-spherical, it follows from Lemma~\ref{lem:Centra}(ii) that $w S'' w\inv = S''$, which implies that $w s'_0 w\inv = s_0$.
Moreover, since   $w$ normalizes $W_{S''}$ and since $(S'')^\perp =\varnothing$, we infer from Lemma~\ref{lem:Centra}(i) that
$w$ must be trivial. Hence $s_0 = s'_0$ and we are done. }
\end{proof}

{
\begin{cor}\label{cor:Inner:intro}
Let $W$ be a finitely generated Coxeter group with Coxeter generating set $S$ and let $\alpha \in \Aut(W)$.
Then $\alpha $ is inner-by-graph if and only if $\alpha(S)$ is angle-compatible with $S$, and $\alpha$ maps every parabolic subgroup to a parabolic subgroup.
\end{cor}

\begin{proof}
{The necessity is obvious. And the sufficiency} follows by applying Proposition~\ref{prop:MainTech} to the Coxeter generating sets $S_1 = S$ and $S_2 = \alpha(S)$.
\end{proof}

}

We {also} deduce the following criterion ensuring that an automorphism is inner.

\begin{cor}\label{cor:Inner}
Let $W$ be a finitely generated Coxeter group with Coxeter generating set $S$ and let $\alpha \in \Aut(W)$.
Then $\alpha $ is inner if and only if $\alpha(S)$ is angle-compatible with $S$, and $\alpha$ maps every parabolic subgroup to a conjugate of itself.
\end{cor}

\begin{proof}
The necessity is trivial. For the sufficiency, suppose that
$S$ and $\alpha(S)$ are two Coxeter generating sets which are reflection-compatible, angle-compatible and parabolic-compatible.
Proposition~\ref{prop:MainTech} ensures that, after replacing $\alpha$ by some appropriate element from the coset $\alpha \Inn(W)$,
we may assume that $\alpha(S) = S$. It then follows from Lemma~\ref{lem:Graph} below that $\alpha $ is inner.
\end{proof}

{

\begin{lem}\label{lem:Graph}
Let $W$ be a finitely generated Coxeter group with Coxeter generating set $S$ and $\alpha \in \Aut(W)$ an automorphism such that $\alpha(S) = S$.
If $\alpha$ maps every parabolic subgroup of $W$ to a conjugate parabolic subgroup, then $\alpha$ is inner.
\end{lem}

\begin{proof}
We first notice that  $\alpha$ preserves each irreducible component of $S$. There is thus no loss of generality in assuming that $S$ is irreducible.

Let $J \subseteq S$ be a subset.
We claim that if $J$ is irreducible and non-spherical (resp. maximal spherical), then $\alpha(J) = J$.
Indeed  $W_J$ is conjugate
to $\alpha(W_J) = W_{\alpha(J)}$ by hypothesis. By \cite{Deodhar}, two irreducible non-spherical (resp. maximal spherical) subsets of $S$ are
conjugate if and only if they coincide.
Thus $J = \alpha(J)$ and the claim stands proven.

Assume now that $S$ is non-spherical. Let $J \subseteq S$ be irreducible non-spherical and minimal with these properties.

Then $\alpha(J) = J$ by the claim above. Moreover, since $S$ is irreducible, we can order the elements of $S \setminus J$, say
$S \setminus J = \{t_1, \dots, t_k\}$, so that $J \cup \{t_1, \dots, t_i\}$
is irreducible (and non-spherical) for all $i$. Applying the claim to each of these sets, we deduce that $\alpha(t_i) = t_i$ for all $i < k$.

Since $J$ is minimal non-spherical, it follows that   for each $s \in J$, the subset $J_s = J \setminus \{s\}$ is contained in some  maximal spherical subset
of $S$ not containing $s$. Applying the claim to such a maximal spherical subset, we infer that $\alpha(J_s) = J_s$. Hence $\alpha(s)= s$.
Thus $\alpha$ acts trivially on $S$ and we are done in this case.

Assume finally that $S$ is spherical. The types of the irreducible finite Coxeter groups admitting a non-trivial graph automorphisms are:
$A_n$ ($n>1$), $D_n$ ($n>3$), $E_6$, $F_4$ and the dihedral groups $I_2(n)$. For types $A_n$ (with $n$ arbitrary),
$D_n$ (with $n$ odd), $E_6$ and $I_2(n)$ (with $n$ odd), the unique non-trivial graph automorphism is inner and realized by the longest element.
For type $F_4$ and $I_2(n)$ with $n$ even, the unique non-trivial graph automorphism swaps the two conjugacy classes of reflections.
Therefore {it does not preserve the conjugacy classes of parabolic subgroups of rank one.}
Finally, for $W$ of type $D_n$ with $n$ even and $\alpha \in \Aut(W)$ a non-trivial graph automorphism, we let $J \subset S$ be one of
the two maximal irreducible proper subsets which is not $\alpha$-invariant. Since $W_J$  is conjugate to $W_{\alpha(J)}$, \cite{Deodhar}
implies that $J = \alpha(J)$, a contradiction.
\end{proof}

\section{Pointwise inner automorphisms of Coxeter groups}
In this section we {{give the  proofs of}  Theorem~\ref{thm:CoxInner} and its corollaries. In fact, we obtain the following result, which is slightly more general than Theorem~\ref{thm:CoxInner}:

\begin{thm}\label{thm:ParabCompat}
Let $S$ and $S'$ be Coxeter generating sets of a  finitely generated Coxeter group $W$.

Then there is an inner automorphism $\alpha \in \Inn(W)$ such that $\alpha(S) = S'$ if and only if the following two conditions are satisfied:
\begin{enumerate}[(1)]
\item For each $J \subseteq S$, there is $J' \subseteq S'$ such that $W_J$ and $W_{J'}$ are conjugate.

\item For all $s, t \in S$ such that $st$ has finite order, there is a pair $s', t' \in S'$ such that $st$ is conjugate to $s't'$.
\end{enumerate}
\end{thm}

\begin{lem}\label{lem:AngleCompat}
Let $S, S'$ be reflection-compatible Coxeter generating sets for a Coxeter group $W$. Suppose that for each spherical pair $\{s, t\} \subseteq S$
there is a spherical pair $\{s', t'\} \subseteq S'$ such that $st$ is conjugate to $s't'$.

Then $S$ and $S'$ are angle-compatible.
\end{lem}

\begin{proof}
{Let $\{s,t\}\subseteq S$ be a spherical pair}. After replacing $S'$ with a conjugate, we may assume that $st = s't'$. By Lemma~\ref{lem:CoxeterElement:Refl}, the parabolic closure of
$st = s't'$ with respect to $S$ (resp. $S'$) is the group $W_{\{s,t\}}$ (resp. $W_{\{s', t'\}}$). On the other hand, Lemma~\ref{lem:Autom} ensures
that $W_{\{s,t\}}$ is parabolic with respect to $S'$ and $W_{\{s', t'\}}$ is parabolic with respect to $S$. It follows that $W_{\{s,t\}} = W_{\{s', t'\}}$.

It is easy to verify that any Coxeter generating pair $s', t'$ of the finite dihedral group $W_{\{s,t\}}$ such that the rotations $st$ and $s't'$ coincide, must
be setwise conjugate to $\{s, t\}$ within the group $W_{\{s,t\}}$.  Therefore $S$ and $S'$ are angle-compatible, as desired.
\end{proof}

\begin{proof}[Proof of Theorem~\ref{thm:ParabCompat}]
That conditions (1) and (2) are necessary is clear. Assume that (1) and (2) hold. Thus $S$ and $S'$ are parabolic-compatible by (1). They are also  angle-compatible by (2), in view of Lemma~\ref{lem:AngleCompat}.
{Hence the conclusion follows from Proposition~\ref{prop:MainTech}.}
\end{proof}

\begin{proof}[Proof of Theorem~\ref{thm:CoxInner}]
{The claim is} immediate from Theorem~\ref{thm:ParabCompat} applied to the Coxeter generating sets $S$ and $S' = \alpha(S)$.
\end{proof}

\begin{proof}[Proof of Corollary~\ref{cor:cryst}]
In view of Theorem~\ref{thm:CoxInner}, it suffices to show that if $\alpha \in \Aut(W)$ satisfies (1) from that theorem, then it also satisfies (2). Given a spherical pair $\{s, t\} \subseteq S$, we know that
$\alpha(W_{\{s, t\}})$ is a spherical parabolic of rank two (by (1) and Lemma~\ref{lem:Autom}). {Thus, after replacing $\alpha$ with some automorphism from the same coset $\alpha Inn(W)$, we can suppose that
$\alpha(W_{\{s, t\}})=W_{\{s',t'\}}$ for some $\{s',t'\}\subseteq S$. Since $S$ is reflection-compatible with $\alpha(S)$, by the assumptions, Lemma \ref{lem:CompatHered} implies that the generating pairs
$\{s',t'\}$ and $\{\alpha(s),\alpha(t)\}$ are reflection-compatible in $W_{\{s',t'\}}$}. It remains to observe that in a finite dihedral group of order $4, 6, 8$ or $12$, any two
reflection-compatible Coxeter generating pairs are automatically angle-compatible. Thus the pairs $\{s, t\}$ and $\{\alpha(s), \alpha(t)\}$ are setwise conjugate, so that $S$ and $\alpha(S)$
are angle-compatible. In particular condition (2) holds, as desired.
\end{proof}

\begin{proof}[Proof of Corollary~\ref{cor:SmallWords}]
Let $S'=\alpha(S)$. Clearly $S'$ is a Coxeter generating set which is reflection-compatible with $S$. Moreover $S$ and $S'$ are angle-compatible by Lemma~\ref{lem:AngleCompat}.

We claim that $\alpha$ maps every parabolic subgroup to a conjugate of itself.
Indeed, let $J \subseteq S$ and let $x_J $ be the product of the elements of $J$ ordered arbitrarily. By hypothesis $\alpha(x_J)$ is conjugate to $x_J$.
Since $\alpha$ maps each reflection to a reflection, it maps a reflection subgroup to a reflection subgroup, and
it follows therefore from Lemma~\ref{lem:CoxeterElement:Refl} that $\alpha(W_{J})$ is conjugate to $W_{J}$.
Therefore $\alpha$ maps every standard  parabolic subgroup to some conjugate of itself.

Thus all the hypotheses of  {Corollary~\ref{cor:Inner}} are satisfied, thereby yielding the claim.
\end{proof}

}

}

\subsection*{Acknowledgements}
We are grateful to Piotr Przytycki for a careful reading of an earlier version of this note.

\end{document}